\documentclass[final]{dmtcs-episciences}
\usepackage[utf8]{inputenc}
\usepackage[english]{babel}
\usepackage[dvipsnames]{xcolor}
\usepackage[round]{natbib}

\usepackage{algorithm2e}
\usepackage{amsfonts}
\usepackage{amsthm}
\usepackage{amssymb}
\usepackage{enumerate}
\usepackage{mathtools, extarrows}
\usepackage{tikz}
\usepackage{graphicx}
\usepackage[margin=1in]{geometry}
\usepackage{hyperref}

\title[Generating Quadrangulations and Operations on Maps]{Generating Plane Quadrangulations and Symmetry-preserving Operations on Maps}
\author{Heidi Van den Camp\affiliationmark{1} \and Brendan D. McKay\affiliationmark{2}}

\affiliation{
	Ghent University, Ghent, Belgium\\
	Australian National University, Canberra, Australia}
\keywords{quadrangulation, generation, symmetry-preserving operation, map}

\newtheorem{theorem}{Theorem}
\newtheorem*{theorem*}{Theorem}
\newtheorem{definition}{Definition}
\newtheorem{lemma}{Lemma}

\newtheorem*{corollary*}{Corollary}
\newtheorem{remark}{Remark}
\newtheorem*{remark*}{Remark}

\newcommand*{\vpointer}[1]{\qquad\quad\vcenter{\hbox{\scalebox{1.5}{$\xlongrightarrow{#1}$}}}\qquad\qquad}

\begin{document}
	
\publicationdata{vol. 26:3}{2024}{9}{10.46298/dmtcs.13080}{2024-02-19; 2024-02-19; 2024-09-09}{2024-09-27}
\maketitle
\begin{abstract}
  ~\\
  Lopsp-operations are operations on maps that are applied locally and are guaranteed to preserve all the orientation-preserving symmetries of maps. Well-known examples of such operations are dual, ambo, truncation, and leapfrog. They are described by plane 3-coloured triangulations with specific properties. We developed and implemented a program that can generate all lopsp-operations of a given size by reducing the problem of generating lopsp-operations to generating all plane quadrangulations that are not necessarily simple. We extended the program plantri to generate these quadrangulations. 
\end{abstract}

\section{Introduction}\label{sec:intro}

Symmetry-preserving operations on polyhedra have been studied for centuries. The operations dual and truncation are probably the most well known, but there are many others, such as ambo, chamfer and leapfrog. These operations are often very intuitive, but they were each described in their own way. For example, truncation can be described as cutting off every vertex of a polyhedron, so that every vertex is replaced by a face. Ambo is similar, but there the cut goes further, so that the new faces share vertices but no edges. Chamfer works differently. It keeps the vertices, but replaces the edges by hexagons. 

Local orientation-preserving symmetry-preserving operations, or lopsp-operations, were first introduced in the supplementary material of \cite{lopsp2017} and were further studied in \cite{lopsp2021,PhDPieter, lopsp_Pieter2021,lopsp2023}. Lopsp-operations describe a class of operations on maps -- i.e.\ embedded graphs -- that preserve all the orientation-preserving symmetries of that map, and they are applied locally. The well-known operations such as dual and truncation can all be described as lopsp-operations. The advantages of this general definition are that it makes it possible to describe every operation in the same way, and that results can be proved for all lopsp-operations instead of only considering some well-known operations separately. 

We have developed and implemented an algorithm to generate all lopsp-operations of a given size.
The first step in generating lopsp-operations is generating plane quadrangulations, allowing parallel edges. We wrote this generator for plane quadrangulations as an extension of the program plantri (\cite{plantri_program,plantri_article}). Plantri was already able to generate all simple plane quadrangulations, as described in \cite{simplequads}. Our algorithm uses the same approach as the one for simple plane quadrangulations, but with an extra extension operation. 

There is a nice correspondence between plane quadrangulations and general plane maps that can be described by the lopsp-operation join. This can be used to generate all plane maps of a given size, allowing loops and parallel edges.

Starting from quadrangulations, it is quite straightforward to generate all lopsp-operations. However, not all lopsp-operations are equally interesting. With every lopsp-operation one can associate a tiling of the plane in a natural way. A lopsp-operation is $\textbf{c}k$ if its associated tiling is $k$-connected and every face has size at least $k$. The most interesting lopsp-operations are the $\textbf{c}3$ lopsp-operations. In fact, these were the only operations included in the original definition from \cite{lopsp2017}. Therefore we have added an option to our program to generate only $\textbf{c}2$ or $\textbf{c}3$ operations. To check the connectivity of a lopsp-operation we use Theorem \ref{thm:2conn_char} and Theorem \ref{thm:3conn_char} to recognise them based on their associated quadrangulation. We also added some filters sooner in the generation process, so that not all quadrangulations have to be generated if only $\textbf{c}2$ or $\textbf{c}3$ lopsp-operations are required.

Some lopsp-operations do not only preserve all orientation-preserving symmetries of a map but also  all orientation-reversing symmetries. These operations, that preserve all symmetries of a map, can be described in a different way that requires less information: They can be described as local symmetry-preserving or lsp-operations. Just like lopsp-operations, they were first defined in \cite{lopsp2017}. We will show how we can recognise if a lopsp-operation can be written as an lsp-operation by looking at its automorphism group. A generator for lsp-operations already exists, see \cite{lspcounts}, which we used to test our generator of lopsp-operations.

This text will be structured as follows: In Section \ref{sec:definitions} we define some important terms such as lopsp-operation and predecoration. Then in Section \ref{sec:quads} we discuss the algorithm for generating quadrangulations, how we tested it, and the connection to the generation of plane maps. In Section \ref{sec:lopsp}, it is explained how to generate lopsp-operations from quadrangulations. We also describe how to recognise if a lopsp-operation is $\textbf{c}2$ or $\textbf{c}3$ and if it is also an lsp-operation. Finally, we give our obtained counts of the number of lopsp-operations of different sizes and with different properties.

\section{Definitions}\label{sec:definitions}

A \emph{map}, also known as embedded graph or 2-cell embedding, is an embedding of a (multi)graph $G$ into an orientable surface $\Sigma$ such that $\Sigma \setminus G$ is the union of a set of open discs, which are the \emph{faces} of the map. It is possible to define maps on surfaces that are not orientable, but in this text, we will always assume that $\Sigma$ is orientable. Combinatorially, a map can be described as a multigraph $G$ together with a function $\sigma$ that maps oriented edges to oriented edges with the same start, such that for each vertex $\sigma$ imposes a cyclic order on the oriented edges starting in that vertex. For a map $M$, let $V_M$ and $E_M$ denote the sets of vertices and edges of $M$ respectively. Let $E'_M$ denote the set of directed edges. We consider two directed edges for each edge of the map, even for loops, i.e.\ edges whose two endpoints are the same vertex. This implies that $|E'_M|=2\cdot|E_M|$. Let $\theta$ be the involution mapping a directed edge to its inverse. The orbit of a directed edge under $\theta\sigma$ is a \emph{face}. The number of directed edges in a face is the \emph{size} of the face. This definition of a face corresponds to the topological notion of face, which is a component of $\Sigma \setminus G$. Every face is a closed walk that is exactly the walk around one of those components. Let $F_M$ denote the set of faces of a map. The \emph{genus} of $M$ is the genus of the orientable surface $\Sigma$ and is equal to $\frac{2-|V_M| + |E_M| -|F_M|}{2}$. A map is \emph{plane} if it has genus 0, i.e.\ it is embedded on a sphere.
A map is a \emph{triangulation} if every face has size 3 and a \emph{quadrangulation} if every face has size 4. A map $M'$ is a \emph{submap} of a map $M$ if its vertices and edges are all in $M$, and its embedding is induced by the embedding of $M$. 

To avoid all confusion about what is meant by an automorphism of maps in this text, we give formal definitions of the different types of automorphisms.

\begin{definition}
	An \emph{orientation-preserving automorphism} of a map $M$ is a bijection $\phi: E'_M \rightarrow E'_M$ such that $\theta\circ\phi = \phi\circ\theta$ and $\sigma\circ\phi = \phi\circ\sigma$. 
	
	An \emph{orientation-reversing automorphism} of a map $M$ is a bijection $\phi: E'_M \rightarrow E'_M$ such that $\theta\circ\phi = \phi\circ\theta$ and $\sigma^{-1}\circ\phi = \phi\circ\sigma$.
	
	An \emph{automorphism} is a bijection that is either an orientation-preserving or an orientation-reversing automorphism. 
	
	The \emph{automorphism group} $Aut(M)$ or symmetry group of a map $M$ is the group consisting of all automorphisms of $M$ with composition as the operation.
	
	The \emph{orientation-preserving automorphism group} $Aut_{OP}(M)$ of a map $M$ is the group consisting of all orientation-preserving automorphisms with composition as the operation.
\end{definition}

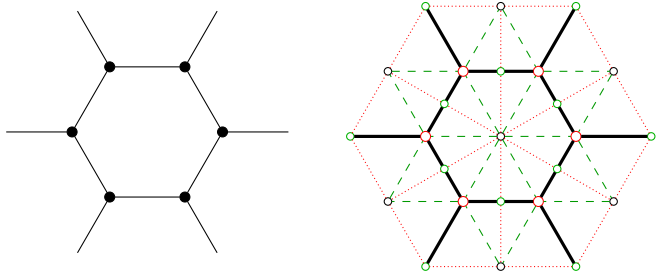
\begin{figure}
	\centering
	\begin{tikzpicture}
\tikzset{normal/.style={shape=circle, fill=black, draw=black, scale=0.4}}
\tikzset{type2/.style={shape=circle, draw=red, scale=0.5}}
\tikzset{noNode/.style={draw=none}}

\node[normal] (v1) at (-1,0) {};
\node[normal] (v4) at (1,0) {};
\node[normal] (v6) at ({cos(120)}, {-sin(120)}) {};
\node[normal] (v3) at ({cos(60)}, {sin(60)}) {};
\node[normal] (v5) at ({cos(60)}, {-sin(60)}) {};
\node[normal] (v2) at ({cos(120)}, {sin(120)}) {};
\draw (v1) -- (v2) -- (v3) -- (v4) -- (v5) -- (v6) -- (v1);
\node [noNode] (v7) at (-2, 0) {};
\node [noNode] (v8) at ({2*cos(120)}, {2*sin(120)}) {};
\node [noNode] (v9) at ({2*cos(60)}, {2*sin(60)}) {};
\node [noNode] (v10) at (2, 0) {};
\node [noNode] (v11) at ({2*cos(60)}, {-2*sin(60)}) {};
\node [noNode] (v12) at ({2*cos(120)}, {-2*sin(120)}) {};
\draw  (v1) edge (v7);
\draw  (v2) edge (v8);
\draw (v3) edge (v9);
\draw (v4) edge (v10);
\draw (v5) edge (v11);
\draw (v6) edge (v12);
\end{tikzpicture}
	\quad
	\begin{tikzpicture}[scale=1]
\tikzset{normal/.style={shape=circle, draw=red, scale=0.4}}
\tikzset{type1/.style={shape=circle, draw=black!30!green, scale=0.3}}
\tikzset{type2/.style={shape=circle, draw=black, scale=0.3}}
\tikzset{noNode/.style={draw=none}}
\tikzset{0edge/.style={draw=red, densely dotted}}
\tikzset{1edge/.style={draw=black!40!green, dashed}}
\tikzset{2edge/.style={very thick}}

\node[normal] (v1) at (-1,0) {};
\node[normal] (v4) at (1,0) {};
\node[normal] (v6) at ({cos(120)}, {-sin(120)}) {};
\node[normal] (v3) at ({cos(60)}, {sin(60)}) {};
\node[normal] (v5) at ({cos(60)}, {-sin(60)}) {};
\node[normal] (v2) at ({cos(120)}, {sin(120)}) {};

\node[type1] (e1) at ({(cos(120)-1)/2}, {sin(60)/2}) {};
\node[type2] (f1) at ({(cos(120)-1)}, {sin(60)}) {};
\node[type1] (e2) at (0, {sin(60)}) {};
\node[type2] (f2) at (0, {2*sin(60)}) {};
\node[type1] (e3) at({(1 - cos(120))/2}, {sin(60)/2}) {};
\node[type2] (f3) at ({(-cos(120)+1)}, {sin(60)}) {};
\node[type1] (e4) at({(1 - cos(120))/2}, {-sin(60)/2}) {};
\node[type2] (f4) at ({(-cos(120)+1)}, {-sin(60)}) {};
\node[type1] (e5) at({(0}, {-sin(60)}) {};
\node[type2] (f5) at (0, {-2*sin(60)}) {};
\node[type1] (e6) at ({(cos(120)-1)/2}, {-sin(60)/2}) {};
\node[type2] (f6) at ({(cos(120)-1)}, {-sin(60)}) {};

\node[type2] (f) at (0, 0) {};

\node [type1] (v7) at (-2, 0) {};
\node [type1] (v8) at ({2*cos(120)}, {2*sin(120)}) {};
\node [type1] (v9) at ({2*cos(60)}, {2*sin(60)}) {};
\node [type1] (v10) at (2, 0) {};
\node [type1] (v11) at ({2*cos(60)}, {-2*sin(60)}) {};
\node [type1] (v12) at ({2*cos(120)}, {-2*sin(120)}) {};

\begin{scope}[0edge]
\draw (f) -- (e1)--(f1);
\draw (f) -- (e2)--(f2);
\draw (f) -- (e3)--(f3);
\draw (f) -- (e4)--(f4);
\draw (f) -- (e5)--(f5);
\draw (f) -- (e6)--(f6);
\draw (f1)--(v8)--(f2)--(v9)--(f3)--(v10)--(f4)--(v11)--(f5)--(v12)--(f6)--(v7)--(f1);
\end{scope}

\begin{scope}[1edge]
\draw (f) edge (v1);
\draw (f) edge (v2);
\draw (f) edge (v3);
\draw (f) edge (v4);
\draw (f) edge (v5);
\draw (f) edge (v6);
\draw (v1)--(f1)--(v2)--(f2)--(v3)--(f3)--(v4)--(f4)--(v5)--(f5)--(v6)--(f6)--(v1);
\end{scope}

\begin{scope}[2edge]
\draw (v1) -- (e1) -- (v2) -- (e2) -- (v3) -- (e3) -- (v4) -- (e4) -- (v5) -- (e5) -- (v6) -- (e6) -- (v1);
\draw  (v1) edge (v7);
\draw  (v2) edge (v8);
\draw (v3) edge (v9);
\draw (v4) edge (v10);
\draw (v5) edge (v11);
\draw (v6) edge (v12);
\end{scope}

\end{tikzpicture}
	\caption{A face of an embedded graph $G$ and the corresponding part of $B_G$. }
	\label{fig:barycentric}
\end{figure}

The barycentric subdivision $B_M$ of a map $M$ is a map with vertex set $V_M\cup E_M \cup F_M$. The vertices corresponding to elements of $V_M$, $E_M$ and $F_M$ are coloured with colours 0, 1 and 2 respectively. There are no edges between vertices of the same colour, and two vertices of $B_M$ are adjacent if the corresponding vertices, edges or faces are incident in $M$. The embedding is as shown in Figure \ref{fig:barycentric}. In all figures in this text, colours 0, 1 and 2 are represented by red, green and black respectively. The map $B_M$ is a properly 3-coloured triangulation and it has the same genus as $M$. The faces of $B_M$ have one vertex of each colour, and are referred to as \emph{chambers} or \emph{flags}.

We use the definition of lopsp-operations from \cite{lopsp2021}. It is slightly different from those in \cite{lopsp2017} and \cite{lopsp_Pieter2021}, but it is more general.

\begin{definition}\label{def:lopsp}
	Let $O$ be a plane map, together with a colouring $c: V_O \rightarrow \{0,1,2\}$ and three special vertices marked as $v_0$, $v_1$, and $v_2$. We say that a vertex has colour $i$ if $c(v)=i$. 
	This map $O$ is a \emph{local orientation-preserving symmetry-preserving operation}, lopsp-operation for short, if the following properties hold:
	\begin{enumerate}[(1)]
		\item $O$ is a triangulation.
		\item There are no edges between vertices of the same colour.
		\item We have
		\[c(v_0),c(v_2)\neq 1\]
		\[c(v_1)=1 \Rightarrow deg(v_1)=2\]
		and for each vertex $v$ different from $v_0$, $v_1$, and $v_2$:
		\[c(v)=1 \Rightarrow deg(v)=4\]
	\end{enumerate}
\end{definition}

\begin{figure}
	\centering
	
	\scalebox{1}{\begin{tikzpicture}[scale=1.5]
\tikzset{type0/.style={shape=circle, draw=red, scale=0.4, fill=white}}
\tikzset{type1/.style={shape=circle, draw=black!30!green, scale=0.4, fill=white}}
\tikzset{type2/.style={shape=circle, draw=black, scale=0.4, fill=white}}
\tikzset{t0/.style={draw=red, densely dotted}}
\tikzset{t1/.style={draw=black!40!green, dashed}}
\tikzset{t2/.style={}}

\node[type0, scale=1.3,  label={[label distance=-0.02cm]90:{${v_0}$}}] (v0) at (0.5, 0) {};
\node[type1, scale=1.3, label={[label distance=-0.06cm]-90:{$v_1$}}] (v1) at (0.25, -0.5) {};
\node[type0, scale=1.3, label={[label distance=-0.02cm]90:{$v_2$}}] (v2) at (0, 1) {};
\node[type1] (e) at (0, 0) {};
\node[type0] (v) at (-0.5, 0) {};
\node[type2] (f) at (1, 0) {};
\node[type1] (E) at (-1, 0) {};

\begin{scope}[t0]
\draw (f)  edge[bend left=90, looseness=1.5] (E) ;
\draw (f)  edge[bend right=85] (E) ;
\draw (f)  edge[bend left=50] (e) ;
\draw (e)  edge[bend left=50] (f) ;
\draw (f)  edge[bend left=30] (v1) ;
\end{scope}

\begin{scope}[t1]
\draw[ultra thick] (v0) -- (f);
\draw[ultra thick] (f)  edge[bend right=45] (v2) ;
\draw (v)  edge[bend left=70] (f) ;
\draw (v)  edge[bend right=80, looseness=1.5] (f) ;
\draw (v)  edge[bend right=55, looseness=1] (f) ;
\end{scope}

\begin{scope}[t2]
\draw[ultra thick] (v0) -- (e) -- (v) ;
\draw (v) -- (E);
\draw (v2)  edge[bend right=45] (E) ;
\draw[ultra thick] (v1)  edge[bend left=30] (v) ;
\end{scope}

\end{tikzpicture}} \qquad
	\scalebox{1}{\begin{tikzpicture}[scale=2.8]
\tikzset{type0/.style={shape=circle, draw=red, scale=0.4, fill=white}}
\tikzset{type1/.style={shape=circle, draw=black!30!green, scale=0.4, fill=white}}
\tikzset{type2/.style={shape=circle, draw=black, scale=0.4, fill=white}}
\tikzset{t0/.style={draw=red, densely dotted}}
\tikzset{t1/.style={draw=black!40!green, dashed}}
\tikzset{t2/.style={}}

\node[type0,   label={[label distance=-0.02cm]90:{$v_2$}}] (v) at (0, 0) {};

%vertices

\node[type1] (e14) at ({cos(50+4*60)*cos(15)/cos(20)}, {sin(50+4*60)*cos(15)/cos(20)}) {};
\node[type0] (v14) at ({cos(40+4*60)*cos(15)/cos(10)}, {sin(40+4*60)*cos(15)/cos(10)}) {};
\node[type1, label={[label distance=-0.02cm]-90:{$v_1$}}] (e24) at ({cos(30+4*60)*cos(15)}, {sin(30+4*60)*cos(15)}) {};
\node[type0] (v24) at ({cos(20+4*60)*cos(15)/cos(10)}, {sin(20+4*60)*cos(15)/cos(10)}) {};
\node[type1] (e34) at ({cos(10+4*60)*cos(15)/cos(20)}, {sin(10+4*60)*cos(15)/cos(20)}) {};
\node[type0,  label={[label distance=-0.0cm]180:{$v_{0,L}$}}] (v34) at ({cos(4*60)*(cos(15)/cos(30))}, {sin(4*60)*(cos(15)/cos(30))}) {};
\node[type1] (e44) at ({cos(40+4*60)*cos(15)/cos(10)/2}, {sin(40+4*60)*cos(15)/cos(10)/2}) {};
\node[type2] (f4) at ({cos(4*60)*(cos(15)/cos(30))/2}, {sin(4*60)*(cos(15)/cos(30))/2}) {};

\node[type2] (f5) at ({cos(5*60)*(cos(15)/cos(30))/2}, {sin(5*60)*(cos(15)/cos(30))/2}) {};
\node[type0,  label={[label distance=-0.0cm]0:{$v_{0,R}$}}] (v35) at ({cos(5*60)*(cos(15)/cos(30))}, {sin(300)*(cos(15)/cos(30))}) {};
%edges

\draw[t2, ultra thick] (v35) -- (e14) -- (v14) --(e24) ;
\draw[t2, ultra thick] (e24) -- (v24) -- (e34) -- (v34) ;
\draw[t2] (v14) -- (e44)--(v);
\draw[t1] (v14) -- (f4) -- (v24);
\draw[t1, ultra thick] (v) -- (f4) -- (v34);
\draw[t0] (e24) -- (f4) -- (e34);
\draw[t0] (f4) -- (e44);
\draw[t0] (e14) -- (f5) -- (e44);
\draw[t1] (v14) -- (f5);

\draw[t1, ultra thick] (v) -- (f5) -- (v35);

\end{tikzpicture}}
	\caption{\label{fig:lopsp_gyro}On the left, the lopsp-operation gyro is shown. The thicker edges are the edges of a cut-path $P$. On the right, $O_P$ is drawn. The two copies of $P$ in $O_P$ are drawn thicker.}
\end{figure}
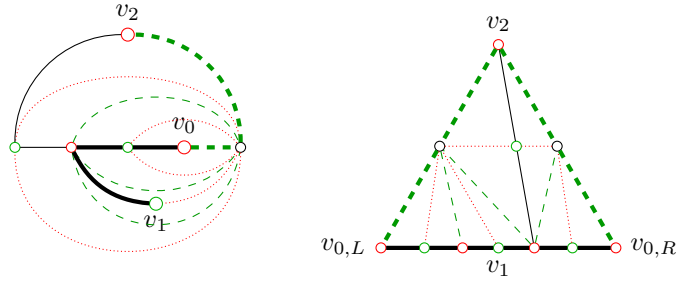

An example of a lopsp-operation is shown in Figure~\ref{fig:lopsp_gyro}. We say that an edge is of \emph{colour $i$} if it is not incident to a vertex of colour $i$. This is well-defined because of (2). Edges of colours 0 and 1 are respectively dotted and dashed in figures. Note that the edges incident
with a vertex are of two different colours, and as every face is a triangle, these colours appear in alternating order in the cyclic order of edges around the vertex. This implies that every
vertex has an even degree. It follows easily from the definition that $O$ is a 2-connected map.
Every face has exactly one vertex and one edge of each colour. We will also call these faces \emph{chambers}. The number of chambers in a lopsp-operation divided by two is the \emph{inflation factor} of the operation.

To apply a lopsp-operation $O$ to a map $M$, choose a path $P$ in $O$ from $v_1$ to $v_0$ to $v_2$. This path is called the \emph{cut-path}. Cut $O$ along this path to create a new map $O_P$ that has a new face, the outer face, that has more than three edges and consists of two copies of $P$. An example of $O_P$ is shown for gyro in Figure~\ref{fig:lopsp_gyro}. The two copies of $v_0$ are marked $v_{0,L}$ and $v_{0,R}$. Now take the barycentric subdivision $B_M$ of $M$, remove the edges of colour 0, and glue a copy of $O_P$ into every face, such that every copy of $v_i$ is glued to a vertex of colour $i$. The edges of $B_M$ are replaced by copies of $P_{v_0,v_1}$ and $P_{v_0,v_2}$. This process is shown in Figure~\ref{fig:gyro_applied}. The result is the barycentric subdivision of a map $O(M)$ (\cite{lopsp2021}), which is the result of applying $M$ to $O$. It is proved in \cite{lopsp2021} that $O(M)$ is independent of the chosen cut-path. Note that the number of chambers in $B_{O(M)}$ is the inflation factor times the number of chambers in $B_M$. This shows that the inflation factor is a measure of how much the operation increases the size of the map.

We can now see how $O$ preserves the orientation-preserving symmetries. It is not only true that $Aut_{OP}(M)\leq Aut_{OP}(O(M))$, every automorphism of $M$ induces an automorphism of $O(M)$ in a natural way: 
Consider the submap $X$ of $B_{O(M)}$ that consists of all the edges that are copies of edges of $P$. As $O(M)$ is constructed by gluing copies of $O_P$ into $B_M$ without the edges of colour 0, the map $X$ is isomorphic to $B_M$ without the edges of colour 0 and with the other edges subdivided as many times as necessary. Every orientation-preserving automorphism of $M$ induces an orientation-preserving automorphism of $X$, and therefore also of $B_{O(M)}$ and $O(M)$. Such an automorphism maps each copy of $O_P$ to another copy of $O_P$. 

\begin{figure}
	\centering
	\scalebox{0.9}{\begin{tikzpicture}[scale=1.8]
\tikzset{type0/.style={shape=circle, draw=red, scale=0.4, fill=white}}
\tikzset{type1/.style={shape=circle, draw=black!30!green, scale=0.4, fill=white}}
\tikzset{type2/.style={shape=circle, draw=black, scale=0.4, fill=white}}
\tikzset{0edge/.style={draw=red, dotted}}
\tikzset{1edge/.style={draw=black!40!green, dashed}}
\tikzset{2edge/.style={}}

\fill[blue!10!white] (0,0) -- ({-cos(15)*tan(30)}, {cos(15)})-- (0, {cos(15)}) -- ({cos(15)*tan(30)}, {cos(15)}) -- cycle;

\node[type2] (v2) at (0, 0) {};

\node[type1] (v1) at (0, {cos(15)}) {};
\node[type1] (v11) at ({cos(15) *cos(30)}, {cos(15)*sin(30)}) {};
\node[type1] (v12) at ({cos(15) *cos(30)}, {-cos(15)*sin(30)}) {};
\node[type1] (v13) at (0, {-cos(15)}) {};
\node[type1] (v14) at ({-cos(15) *cos(30)}, {-cos(15)*sin(30)}) {};
\node[type1] (v15) at ({-cos(15) *cos(30)}, {cos(15)*sin(30)}) {};

\node[type0] (v0) at ({cos(15)*tan(30)}, {cos(15)}) {};
\node[type0] (v01) at ({cos(15)/sin(60)}, 0) {};
\node[type0] (v02) at ({cos(15)*tan(30)}, {-cos(15)}) {};
\node[type0] (v03) at ({-cos(15)*tan(30)}, {-cos(15)}) {};
\node[type0] (v04) at ({-cos(15)/sin(60)}, 0) {};
\node[type0] (v05) at ({-cos(15)*tan(30)}, {cos(15)}) {};

\begin{scope}[0edge]
\draw[] (v1) -- (v2) ;
\draw (v2) -- (v11);
\draw (v12) -- (v2) -- (v13);
\draw (v14) -- (v2) -- (v15);
\end{scope}

\begin{scope}[2edge]
\draw[ultra thick] (v05)--(v1) -- (v0);
\draw (v0) --(v11)--(v01) -- (v12) -- (v02) -- (v13) 
-- (v03) -- (v14) -- (v04) -- (v15) -- (v05);
\end{scope}

\begin{scope}[1edge]
\draw[ultra thick] (v0) -- (v2) --(v05);
\draw (v2) -- (v01);
\draw (v02) -- (v2) -- (v03);
\draw (v04) -- (v2);
\end{scope}

\end{tikzpicture}} 
	\qquad\qquad
	\scalebox{0.9}{\begin{tikzpicture}[scale=1.8]
\tikzset{type0/.style={shape=circle, draw=red, scale=0.4, fill=white}}
\tikzset{type1/.style={shape=circle, draw=black!30!green, scale=0.4, fill=white}}
\tikzset{type2/.style={shape=circle, draw=black, scale=0.4, fill=white}}
\tikzset{t0/.style={draw=red, densely dotted}}
\tikzset{t1/.style={draw=black!40!green, dashed}}
\tikzset{t2/.style={}}

\fill[blue!10!white] (0,0) -- ({-cos(15)*tan(30)}, {cos(15)}) -- (0, {cos(15)}) -- ({cos(15)*tan(30)}, {cos(15)}) -- cycle;

\node[type0] (v) at (0, 0) {};

%vertices
\foreach \i in {0,...,5}
{
\node[type1] (e1\i) at ({cos(50+\i*60)*cos(15)/cos(20)}, {sin(50+\i*60)*cos(15)/cos(20)}) {};
\node[type0] (v1\i) at ({cos(40+\i*60)*cos(15)/cos(10)}, {sin(40+\i*60)*cos(15)/cos(10)}) {};
\node[type1] (e2\i) at ({cos(30+\i*60)*cos(15)}, {sin(30+\i*60)*cos(15)}) {};
\node[type0] (v2\i) at ({cos(20+\i*60)*cos(15)/cos(10)}, {sin(20+\i*60)*cos(15)/cos(10)}) {};
\node[type1] (e3\i) at ({cos(10+\i*60)*cos(15)/cos(20)}, {sin(10+\i*60)*cos(15)/cos(20)}) {};
\node[type0] (v3\i) at ({cos(\i*60)*(cos(15)/cos(30))}, {sin(\i*60)*(cos(15)/cos(30))}) {};
\node[type1] (e4\i) at ({cos(40+\i*60)*cos(15)/cos(10)/2}, {sin(40+\i*60)*cos(15)/cos(10)/2}) {};
\node[type2] (f\i) at ({cos(\i*60)*(cos(15)/cos(30))/2}, {sin(\i*60)*(cos(15)/cos(30))/2}) {};
}

%edges
\foreach \j / \i in {1/0,4/3,5/4,0/5}
{
\draw[t2] (v3\j) -- (e1\i) -- (v1\i) --(e2\i) -- (v2\i) -- (e3\i) -- (v3\i) ;
\draw[t2] (v1\i) -- (e4\i)--(v);
\draw[t1] (v1\i) -- (f\i) -- (v2\i);
\draw[t1] (v) -- (f\i) -- (v3\i);
\draw[t0] (e2\i) -- (f\i) -- (e3\i);
\draw[t0] (f\i) -- (e4\i);
\draw[t0] (e1\i) -- (f\j) -- (e4\i);
\draw[t1] (v1\i) -- (f\j);
}

%bluechambers
\foreach \j / \i in {2/1}
{
\begin{scope}[very thick]
\draw[t2,very thick] (v3\j) -- (e1\i) -- (v1\i) --(e2\i) -- (v2\i) -- (e3\i) -- (v3\i) ;
\draw[t2,very thick] (v1\i) -- (e4\i)--(v);
\draw[t1,very thick] (v1\i) -- (f\i) -- (v2\i);
\draw[t1,very thick] (v) -- (f\i) -- (v3\i);
\draw[t0,very thick] (e2\i) -- (f\i) -- (e3\i);
\draw[t0,very thick] (f\i) -- (e4\i);
\draw[t0,very thick] (e1\i) -- (f\j) -- (e4\i);
\draw[t1,very thick] (v1\i) -- (f\j);
\end{scope}
}

\foreach \j / \i in {3/2}
{
\draw[t2] (v3\j) -- (e1\i) -- (v1\i) --(e2\i) -- (v2\i) -- (e3\i) -- (v3\i) ;
\draw[t2] (v1\i) -- (e4\i)--(v);
\draw[t1] (v1\i) -- (f\i) -- (v2\i);
\draw[t1, very thick] (v) -- (f2) -- (v32);
\draw[t0] (e2\i) -- (f\i) -- (e3\i);
\draw[t0] (f\i) -- (e4\i);
\draw[t0] (e1\i) -- (f\j) -- (e4\i);
\draw[t1] (v1\i) -- (f\j);
}

\end{tikzpicture}}
	
	\caption{\label{fig:gyro_applied}On the left, the barycentric subdivision of a hexagonal face is shown. On the right, the lopsp-operation gyro is applied to it. The blue shaded area shows one double chamber.}
\end{figure}
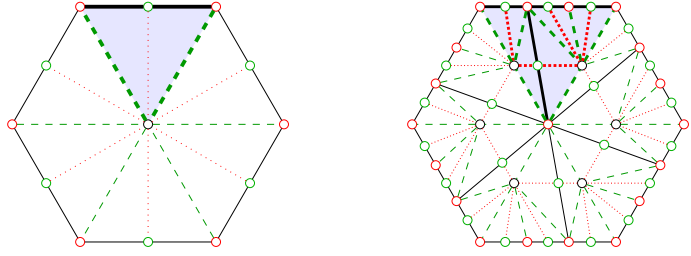

\begin{definition}\label{def:predecoration}
	Let $O$ be a lopsp-operation. If $c(v_1)\neq 1$, then the \emph{predecoration} $Q(O)$ is the submap of $O$ consisting of all vertices of colours 0 and 2 and all edges of colour 1. If $c(v_1)= 1$, then $Q(O)$ is also that map, but together with $v_1$ and the edge of colour 2 incident with $v_1$. 
\end{definition}

It follows from the degree restrictions in Definition \ref{def:lopsp} that every predecoration is a plane quadrangulation.

\section{Generating quadrangulations} \label{sec:quads}
Generating lopsp-operations will be done by first generating all predecorations, which are all plane quadrangulations. In this section we describe our algorithm for generating quadrangulations, and how our implementation was tested. We also show how this generator can be used to generate all plane maps.

\subsection{Generation}
Let $\mathcal{Q}$ be the class of all plane quadrangulations, parallel edges allowed, and let $\mathcal{Q}_n$ be the class of all plane quadrangulations with $n$ vertices. Note that loops cannot occur as all plane quadrangulations are bipartite. The algorithm works in the following way: Start with the smallest possible quadrangulation(s). Then recursively apply \emph{extensions} that increase the number of vertices, until the required number of vertices is reached. Every extension must turn a plane quadrangulation into another plane quadrangulation. We say that a set of extensions \emph{generates} a class of maps from a set of base maps if every map in the class can be obtained by repeatedly applying extensions from the set to one of the base maps.

To ensure that no isomorphic maps are generated, the \emph{canonical construction path method}, described in \cite{isomorphfree_generation} is used. The basic idea of this method is that for every intermediate map that is constructed, a canonical ancestor is determined. If the map was not generated from that ancestor, it is rejected. Otherwise it is accepted and extensions may be applied to it. In this way, each map can only be generated once. 
This algorithm is the same as the one used to generate simple quadrangulations in \cite{simplequads}, but we use one extension more. All other generators in plantri also use this method (\cite{plantri_program,plantri_article}). Theorem \ref{thm:extensions_sufficient} proves that the extensions in Figure \ref{fig:extensions} are sufficient to generate all plane quadrangulations. For the generation of simple plane quadrangulations only extensions $P_0$ and $P_1$ were used. The inverse of an extension is called a \emph{reduction}.

\begin{theorem}\label{thm:extensions_sufficient}
	Extensions $D_1$, $P_0$, and $P_1$ generate $\mathcal{Q}$ from the path graph with 3 vertices.
\end{theorem}
\begin{proof}
	Let $Q$ be any quadrangulation different from the path graph with 3 vertices. All three extensions increase the number of vertices by one, so by induction it suffices to prove that $Q$ can be constructed from a smaller quadrangulation with extension $D_1$, $P_0$, or $P_1$.
	
	Assume first that $Q$ has a vertex $v$ of degree 1. The vertex adjacent to $v$ cannot have degree 1 for then $Q$ would be the path-graph with 2 vertices, which is not a quadrangulation. If it has degree 2 then as every face has size four, $Q$ is the path graph with 3 vertices which we assumed is not the case. If it has degree 3 or higher then $v$ is in a face of size four on which reduction $D_1$ can be applied.
	
	Now assume that every vertex of $Q$ has degree at least 2. We will construct a submap $Q'$ of $Q$ that is a simple plane quadrangulation. If $Q$ has no cycles of length 2, let $Q'$ be $Q$. Otherwise, let $c$ be an innermost 2-cycle in $Q$, i.e.\ a cycle that has no other 2-cycles on one of its sides, which we will call the inside. Both vertices of $c$ have a neighbour on the inside. Otherwise, as every face has size 4, there would either be a vertex of degree 1 on the inside, or there would be two edges between a vertex on the inside and a vertex of $c$ and $c$ would not be innermost. Let $Q'$ be the map consisting of all vertices and edges on the inside of $c$, together with one edge of $c$. As each vertex of $c$ has a neighbour on the inside, $Q'$ has no vertices of degree 1. This map $Q'$ is a simple plane quadrangulation, and each of its faces has a corresponding face in $Q$. With the Euler formula it follows that $|E_{Q'}|=2|V_{Q'}| - 4$, and therefore \[\sum_{v\in V_{Q'}}deg(v) = 2|E_{Q'}| = 4|V_{Q'}| - 8.\] 
	This implies that there are at least 4 vertices of degree at most 3. It follows that there is a vertex of degree at most 3 that is not a vertex of $c$ and therefore it also has degree 2 or 3 in $Q$. We can apply reduction $P_0$ or $P_1$ on a face of $Q$ containing that vertex. Such a face is simple because $Q'$ is simple.
\end{proof}

\begin{figure}
	\begin{center}
		$\vcenter{\hbox{\scalebox{1.5}{\begin{tikzpicture}[scale=0.8]
\tikzset{every node/.style={draw=black,fill=black, shape=circle, scale=0.4}}

\node (a) at (-1,0) {};
\node (b) at (1,0) {};

\begin{scope}
\draw (a) edge (b);
\end{scope}

\draw (a) -- (-1.4,-0.1) -- (-1.4,0.1) -- (a);
\draw (b) -- (1.4,-0.1) -- (1.4,0.1) -- (b);
\draw[draw=none] (-1.5,0) -- +(-0.4,-0.1)
-- +(-0.4,0.1) -- (-1.5,0);
\draw[draw=none] (1.5,0) -- +(0.4,-0.1)
-- +(0.4,0.1) -- (1.5,0);

\end{tikzpicture}}}} 
		\vpointer{D_1}
		\vcenter{\hbox{\scalebox{1.5}{\begin{tikzpicture}[scale=0.8]
\tikzset{every node/.style={draw=black,fill=black, shape=circle, scale=0.4}}

\node (a) at (-1,0) {};
\node (b) at (1,0) {};
\node (c) at (0,0) {};

\begin{scope}
\draw (a) edge[in = 110, out = 70, looseness = 1.3] (b);
\draw (a) edge[in = -110, out = -70, looseness = 1.3] (b);
\end{scope}

\draw (a)--(c);
\draw (a) -- (-1.4,-0.1) -- (-1.4,0.1) -- (a);
\draw (b) -- (1.4,-0.1) -- (1.4,0.1) -- (b);

\draw[draw=none] (-1.5,0) -- +(-0.4,-0.1)
-- +(-0.4,0.1) -- (-1.5,0);
\draw[draw=none] (1.5,0) -- +(0.4,-0.1)
-- +(0.4,0.1) -- (1.5,0);
\end{tikzpicture}}}}$
		
		\vspace{0.5cm}
		
		$\vcenter{\hbox{\scalebox{1.5}{\begin{tikzpicture}[scale=0.8]
\tikzset{every node/.style={draw=black,fill=black, shape=circle, scale=0.4}}

\node (a) at (-1.5,0) {};
\node (b) at (1.5,0) {};
\node (c) at (0,0) {};

\begin{scope}
\draw (a) --(c) -- (b);
\end{scope}

\draw (a) -- +(-0.4,-0.1)
		  -- +(-0.4,0.1) -- (a);
\draw (b) -- +(0.4,-0.1)
		  -- +(0.4,0.1) -- (b);
\draw (c) -- +(0.1,-0.4)
		  -- +(-0.1,-0.4) -- (c);
\end{tikzpicture}}}} 
		\vpointer{P_0}
		\vcenter{\hbox{\scalebox{1.5}{\begin{tikzpicture}[scale=0.8]
\tikzset{every node/.style={draw=black,fill=black, shape=circle, scale=0.4}}

\node (a) at (-1.5,0) {};
\node (b) at (1.5,0) {};
\node (c) at (0,-0.5) {};
\node (d) at (0,0.5) {};

\begin{scope}
\draw (a) --(c) -- (b) -- (d) -- (a);
\end{scope}

\draw (a) -- +(-0.4,-0.1)
		  -- +(-0.4,0.1) -- (a);
\draw (b) -- +(0.4,-0.1)
		  -- +(0.4,0.1) -- (b);
\draw (c) -- +(0.1,-0.4)
		  -- +(-0.1,-0.4) -- (c);
\end{tikzpicture}}}}$
		
		\vspace{0.5cm}
		
		$\vcenter{\hbox{\scalebox{1.5}{\begin{tikzpicture}[scale=0.8]
\tikzset{every node/.style={draw=black,fill=black, shape=circle, scale=0.4}}

\node (a) at (-1.5,0) {};
\node (b) at (1.5,0) {};
\node (c) at (0,0) {};

\begin{scope}
\draw (a) --(c) -- (b) ;
\end{scope}

\draw (a) -- +(-0.4,-0.1)
		  -- +(-0.4,0.1) -- (a);
\draw (b) -- +(0.4,-0.1)
		  -- +(0.4,0.1) -- (b);
\draw (c) -- +(0.1,-0.4)
		  -- +(-0.1,-0.4) -- (c);
\draw (c) -- +(0,0.4);
\draw (c) -- +(-0.2,-0.4);
\end{tikzpicture}}}} 
		\vpointer{P_1}
		\vcenter{\hbox{\scalebox{1.5}{\begin{tikzpicture}[scale=0.8]
\tikzset{every node/.style={draw=black,fill=black, shape=circle, scale=0.4}}

\node (a) at (-1.5,0) {};
\node (b) at (1.5,0) {};
\node (c) at (0,-0.5) {};
\node (d) at (0,0.5) {};

\begin{scope}
\draw (a) --(c) -- (b) -- (d) -- (a);
\end{scope}

\draw (a) -- +(-0.4,-0.1)
		  -- +(-0.4,0.1) -- (a);
\draw (b) -- +(0.4,-0.1)
		  -- +(0.4,0.1) -- (b);
\draw (c) -- +(0.1,-0.4)
		  -- +(-0.1,-0.4) -- (c);
\draw (d) -- +(0,0.4);
\draw (c) -- +(-0.2,-0.4);
\end{tikzpicture}}}}$
	\end{center}
	\caption{Three different extensions that can be used to generate all plane quadrangulations. Every vertex and edge in each of the drawings is distinct. A triangle indicates that there may be an arbitrary number of edges incident to the vertex in that place.}
	\label{fig:extensions}
\end{figure}
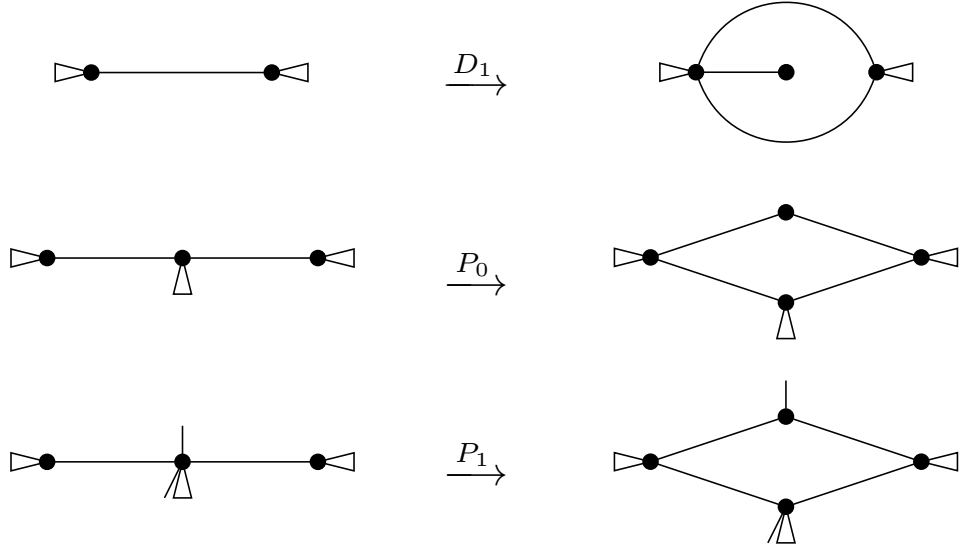

\subsection{Testing}

\begin{table}
	\centering
	\begin{tabular}{ccc|r}
		$|V|$ & $|E|$ & $|F|$& $|\mathcal{Q}_{|V|}|$ \\
		\hline
		3 & 2& 1 & 1 \\
		
		4& 4& 2& 3 \\
		
		5& 6& 3& 7 \\
		
		6& 8& 4& 30 \\
		
		7& 10& 5& 124 \\
		
		8& 12& 6& 733 \\
		
		9& 14& 7& 4 586 \\
		
		10& 16& 8& 33 373 \\
		
		11& 18&9 & 259 434 \\
		
		12& 20& 10& 2 152 298 \\
		
		13& 22&11 & 18 615 182 \\
		
		14& 24& 12& 166 544 071 \\
		
		15& 26&13 & 1 528 659 536 \\
		
		16& 28&14 & 14 328 433 429 \\
		
		17& 30&15 & 136 649 176 084 \\
		
		18& 32&16 & 1 322 594 487 342 \\
		
		19& 34&17 & 12 965 736 092 988 \\
		
		20& 36&18 & 128 543 259 338 048 \\
		
	\end{tabular}
	\caption{This table shows the number of non-isomorphic plane quadrangulations with a given number of vertices, edges, or faces.}\label{tab:quadcounts}
\end{table}

We implemented this algorithm as an extension of plantri and with it we calculated the number of non-isomorphic plane quadrangulations up to 20 vertices, as shown in Table \ref{tab:quadcounts}. 
To check the program, different tests were used. The counts we obtained were compared with existing counts of quadrangulations. In the literature we only found counts for $n\leq 12$ in the most general case, see \cite{quadcounts_knots, quadcounts_equilibrium}, so this check could not be done for higher numbers of vertices. There are however theoretical formulas for the number of quadrangulations with $n$ vertices up to orientation-preserving automorphism, see \cite{quadcountsOP}, and the number of rooted quadrangulations, see \cite{quadcounts_rooted}. As the program can also generate these, our counts were checked against these theoretical values.

For $n\leq 17$ it was checked that our program does not output quadrangulations that are isomorphic to each other. 
Another check was done for $n\leq 15$. This check is based on the \emph{rotations} shown in Figure~\ref{fig:flips}. These are local changes that can be made to quadrangulations to get other quadrangulations with the same number of vertices. We applied all possible rotations of type $A$ and $C$ to every generated quadrangulation and checked if we had also generated the result. With Theorem \ref{thm:general_quad_flips} this check proves that all quadrangulations were generated.
\begin{figure}
	\begin{center}
		$\vcenter{\hbox{\scalebox{1.5}{\begin{tikzpicture}
\tikzset{every node/.style={draw=black,fill=black, shape=circle, scale=0.4}}

\node (a) at (0,1) {};
\node (b) at (1,1) {};
\node (c) at (-0.5,0) {};
\node (d) at (1.5,0) {};
\node (e) at (0,-1) {};
\node (f) at (1,-1) {};

\begin{scope}
\draw (a)--(b)--(d)--(f)--(e)--(c)--(a);
\end{scope}

\draw[dashed] (c)--(d);
\end{tikzpicture}}}} 
		\vpointer{A}
		\vcenter{\hbox{\scalebox{1.5}{\begin{tikzpicture}
\tikzset{every node/.style={draw=black,fill=black, shape=circle, scale=0.4}}

\node (a) at (0,1) {};
\node (b) at (1,1) {};
\node (c) at (-0.5,0) {};
\node (d) at (1.5,0) {};
\node (e) at (0,-1) {};
\node (f) at (1,-1) {};

\begin{scope}
\draw (a)--(b)--(d)--(f)--(e)--(c)--(a);
\end{scope}

\draw[red] (a)--(f);
\end{tikzpicture}}}}$
		
		\vspace{0.5cm}
		
		$\vcenter{\hbox{\scalebox{1.5}{\begin{tikzpicture}
\tikzset{every node/.style={draw=black,fill=black, shape=circle, scale=0.4}}

\node (a) at (0,1.7) {};
\node (b) at (1.7,1.7) {};
\node (c) at (-0,0) {};
\node (d) at (1.7,0) {};
\node (e) at (0.85,0.85) {};

\node[draw=none, fill=none] (_) at (1.85,2) {};
\node[draw=none, fill=none] (_) at (-0.15,0) {};

\begin{scope}
\draw (a)--(b)--(d)--(c)--(a);
\end{scope}

\draw[dashed] (a)--(e)--(d);
\end{tikzpicture}}}} 
		\vpointer{B}
		\vcenter{\hbox{\scalebox{1.5}{\begin{tikzpicture}
\tikzset{every node/.style={draw=black,fill=black, shape=circle, scale=0.4}}

\node (a) at (0,1.7) {};
\node (b) at (1.7,1.7) {};
\node (c) at (-0,0) {};
\node (d) at (1.7,0) {};
\node (e) at (0.85,0.85) {};

\node[draw=none, fill=none] (_) at (1.85,2) {};
\node[draw=none, fill=none] (_) at (-0.15,0) {};

\begin{scope}
\draw (a)--(b)--(d)--(c)--(a);
\end{scope}

\draw[red] (b)--(e)--(c);
\end{tikzpicture}}}}$
		
		\vspace{0.5cm}
		
		$\vcenter{\hbox{\scalebox{1.5}{\begin{tikzpicture}
\tikzset{every node/.style={draw=black,fill=black, shape=circle, scale=0.4}}

\node (a) at (-1,0) {};
\node (b) at (1,0) {};
\node (c) at (0,0) {};

\begin{scope}
\draw (a) edge[in = 110, out = 70, looseness = 1.3] (b);
\draw (a) edge[in = -110, out = -70, looseness = 1.3] (b);
\end{scope}

\draw[dashed] (a)--(c);
\end{tikzpicture}}}} 
		\vpointer{C}
		\vcenter{\hbox{\scalebox{1.5}{\begin{tikzpicture}
\tikzset{every node/.style={draw=black,fill=black, shape=circle, scale=0.4}}

\node (a) at (-1,0) {};
\node (b) at (1,0) {};
\node (c) at (0,0) {};

\begin{scope}
\draw (a) edge[in = 110, out = 70, looseness = 1.3] (b);
\draw (a) edge[in = -110, out = -70, looseness = 1.3] (b);
\end{scope}

\draw[red] (b)--(c);
\end{tikzpicture}}}}$
	\end{center}
	\caption{Three different rotations that can be applied to quadrangulations. The vertices drawn are not necessarily distinct.}
	\label{fig:flips}
\end{figure}
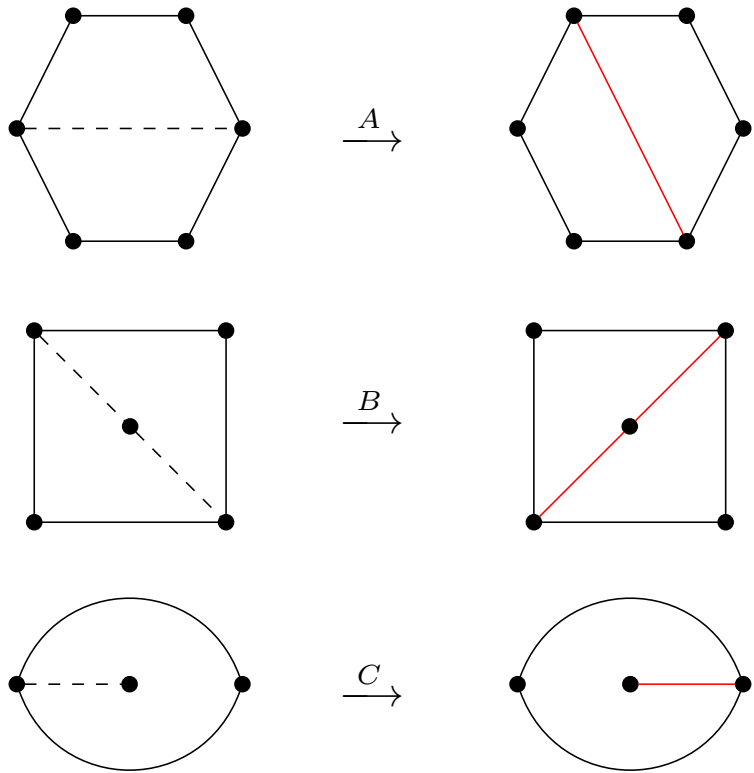

\begin{theorem}\label{thm:general_quad_flips} 
	Let $n\in \mathbb{N}$. Any two plane quadrangulations $Q_1\in \mathcal{Q}_n$ and $Q_2\in \mathcal{Q}_n$ can be transformed into each other by a number of applications of $A$ and $C$.	
\end{theorem}
\begin{proof}
	Operation $B$ can be replaced by applying $A$, then $C$ and then $A$ again as shown in Figure~\ref{fig:B_equals_ACA}. This implies that we can also use operation $B$ in our arguments.
	
	It is proved in \cite{flips, flips2} that any two simple, plane quadrangulations can be transformed into each other by applying $A$ and $C$. Therefore it suffices to prove that every quadrangulation with multiple edges can be transformed into a simple quadrangulation using rotations $A$ and $C$. We will prove that if a quadrangulation has parallel edges, then by applying a number of operations $A$ and $C$ we can get a quadrangulation with strictly fewer parallel edges. 
	
	Let $Q$ be a quadrangulation with parallel edges $e_1$ and $e_2$ that are incident to vertices $x$ and $y$ such that the cycle $c$ with edges $e_1$ and $e_2$ is an innermost 2-cycle. 
	On the outside of $c$ there is at least one vertex that is not $x$ or $y$, but it is adjacent to at least one of them. Assume w.l.o.g. that there is a vertex on the outside of $c$ that is adjacent to $x$. If $y$ has no neighbour on the inside of $c$, then there can be only one vertex on the inside of $c$ and it has degree one. In any other case, there would be parallel edges on the inside of $c$, which is impossible as $c$ is innermost. It follows that either $y$ has a neighbour on the inside of $c$, or we can apply $C$ so that $y$ has a neighbour on the inside of $c$. 
	
	Assume w.l.o.g.\ that the edges incident with $x$ on the outside of $c$ appear after $e_2$ and before $e_1$ in the cyclic order around $x$. 
	Now we can apply $A$ to $e_1$ such that it is replaced by an edge from a vertex $v$ adjacent to $x$ on the outside of $c$ and a vertex $w$ adjacent to $y$ on the inside of $c$. This is not a parallel edge, as $v$ and $w$ are on different sides of $c$ so they cannot be adjacent in $Q$. It follows that the number of parallel edges has decreased by possibly applying $C$ and then applying $A$.
\end{proof}

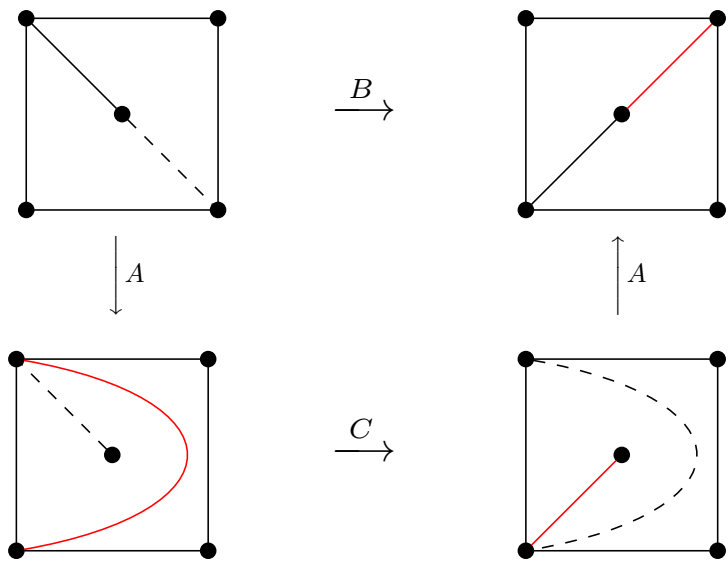
\begin{figure}
	\begin{center}
		$\vcenter{\hbox{\scalebox{1.5}{\begin{tikzpicture}
\tikzset{every node/.style={draw=black,fill=black, shape=circle, scale=0.4}}

\node (a) at (0,1.7) {};
\node (b) at (1.7,1.7) {};
\node (c) at (-0,0) {};
\node (d) at (1.7,0) {};
\node (e) at (0.85,0.85) {};

\node[draw=none, fill=none] (_) at (1.85,2) {};
\node[draw=none, fill=none] (_) at (-0.15,0) {};

\begin{scope}
\draw (e)--(a)--(b)--(d)--(c)--(a);
\end{scope}

\draw[dashed] (e)--(d);
\end{tikzpicture}}}} 
		\vpointer{B}
		\vcenter{\hbox{\scalebox{1.5}{\begin{tikzpicture}
\tikzset{every node/.style={draw=black,fill=black, shape=circle, scale=0.4}}

\node (a) at (0,1.7) {};
\node (b) at (1.7,1.7) {};
\node (c) at (-0,0) {};
\node (d) at (1.7,0) {};
\node (e) at (0.85,0.85) {};

\node[draw=none, fill=none] (_) at (1.85,2) {};
\node[draw=none, fill=none] (_) at (-0.15,0) {};

\begin{scope}
\draw (a)--(b)--(d)--(c)--(a)
(c)--(e);
\end{scope}

\draw[red] (e)--(b);
\end{tikzpicture}}}}$
		
		\vspace{0.2cm}
		
		$\Bigg\downarrow A$
		\hspace{6cm}
		$\Bigg\uparrow A$
		
		$\vcenter{\hbox{\scalebox{1.5}{\begin{tikzpicture}
	\tikzset{every node/.style={draw=black,fill=black, shape=circle, scale=0.4}}
	
	\node (a) at (0,1.7) {};
	\node (b) at (1.7,1.7) {};
	\node (c) at (-0,0) {};
	\node (d) at (1.7,0) {};
	\node (e) at (0.85,0.85) {};
	
	\node[draw=none, fill=none] (_) at (1.85,2) {};
	\node[draw=none, fill=none] (_) at (-0.15,0) {};
	
	\begin{scope}
		\draw (a)--(b)--(d)--(c)--(a);
		\draw[dashed] (e)--(a);
	\end{scope}
	
	\draw (a) edge[in=10, out = -10, looseness=3,red] (c);
\end{tikzpicture}}}} 
		\vpointer{C}
		\vcenter{\hbox{\scalebox{1.5}{\begin{tikzpicture}
	\tikzset{every node/.style={draw=black,fill=black, shape=circle, scale=0.4}}
	
	\node (a) at (0,1.7) {};
	\node (b) at (1.7,1.7) {};
	\node (c) at (-0,0) {};
	\node (d) at (1.7,0) {};
	\node (e) at (0.85,0.85) {};
	
	\node[draw=none, fill=none] (_) at (1.85,2) {};
	\node[draw=none, fill=none] (_) at (-0.15,0) {};
	
	\begin{scope}
		\draw (a)--(b)--(d)--(c)--(a);
		\draw[red] (e)--(c);
	\end{scope}
	
	\draw (a) edge[in=10, out = -10, looseness=3, dashed] (c);
\end{tikzpicture}}}}$
	\end{center}
	\caption{Rotation $B$ can be replaced by applications of $A$, $C$, and $A$ again.}
	\label{fig:B_equals_ACA}
\end{figure}

\subsection{Generating plane maps} 
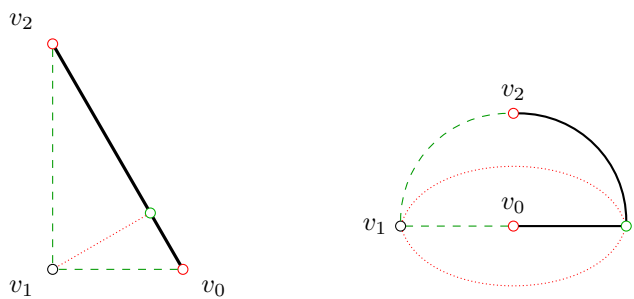
\begin{figure}
	\begin{center}
		\begin{tikzpicture}[scale = 3]
\tikzset{type0/.style={shape=circle, draw=red, scale=0.4, fill=white}}
\tikzset{type1/.style={shape=circle, draw=black!30!green, scale=0.4, fill=white}}
\tikzset{type2/.style={shape=circle, draw=black, scale=0.4, fill=white}}
\tikzset{noNode/.style={draw=none}}
\tikzset{0edge/.style={draw=red, densely dotted}}
\tikzset{1edge/.style={draw=black!40!green, dashed}}
\tikzset{2edge/.style={very thick}}

\node[type0] (v2) at (0, 1) {};
\node[type2] (v1) at (0, 0) {};
\node[type0] (v0) at ({tan(30)}, 0) {};
\node[type1] (e) at ({cos(30)*sin(30)},{sin(30)^2}) {};

\begin{scope}[2edge]
\draw (e) edge (v2);
\draw (e) edge (v0);
\end{scope}

\begin{scope}[1edge]
\draw (v0) edge (v1);
\draw (v2) edge (v1);
\end{scope}

\draw[red, densely dotted] (e) edge (v1);

\node[label=150:{$v_2$}] at (0, 1) {};
\node[label=190:$v_1$] at (0, 0) {};
\node[label={-10:$v_0$}] at ({tan(30)}, 0) {};
\end{tikzpicture} \qquad \qquad
		\begin{tikzpicture}[scale=2]
\tikzset{type0/.style={shape=circle, draw=red, scale=0.4, fill=white}}
\tikzset{type1/.style={shape=circle, draw=black!30!green, scale=0.4, fill=white}}
\tikzset{type2/.style={shape=circle, draw=black, scale=0.4, fill=white}}
\tikzset{t0/.style={draw=red, densely dotted}}
\tikzset{t1/.style={draw=black!40!green, dashed}}
\tikzset{t2/.style={thick}}

\node[type0, label={[label distance=-0.02cm]90:$v_0$}] (v0) at (0, 0) {};
\node[type2, label={[label distance=-0.02cm]-180:$v_1$}] (v1) at (-0.75, 0) {};
\node[type0, label={[label distance=-0.02cm]90:$v_2$}] (v2) at (0, 0.75) {};
\node[type1] (e) at (0.75, 0) {};

\begin{scope}[t0]
\draw (v1) edge[in=110, out = 70, looseness=0.9] (e);
\draw (v1) edge[in=-110, out = -70, looseness=0.9] (e);
\end{scope}

\begin{scope}[t1]
\draw (v0) -- (v1);
\draw (v1) edge[in=180, out = 90] (v2);
\end{scope}

\begin{scope}[t2]
\draw (v0) -- (e);
\draw (e) edge[in=0, out = 90] (v2);
\end{scope}

\end{tikzpicture}
	\end{center}
	\caption{The operation join as an lsp-operation and as a lopsp-operation.}
	\label{fig:join}
\end{figure}

In this short section we use the lopsp-operation join. This operation is given in Figure \ref{fig:join} as an lsp-operation (Definition~\ref{def:lsp}) and as a lopsp-operation. For any map $M$, $\text{join}(M)$ can be described as the barycentric subdivision of $M$ without the vertices of colour 1 and ignoring the colours. It is also known as the radial graph of $M$.

We can apply the lopsp-operation join to any plane map to get a plane quadrangulation. Given a quadrangulation, we can also determine to which map(s) join can be applied to get that quadrangulation. This connection between maps and quadrangulations is formalised in Lemma~\ref{lem:map_quad}.

\begin{lemma}\label{lem:map_quad}
	There is a bijection between the set of all plane maps and the set of all plane quadrangulations with a fixed bipartition class.
\end{lemma}
\begin{proof}
	Let $\phi$ be the mapping that maps a map $M$ to the quadrangulation $\text{join}(M)$ with the bipartition class consisting of the original vertices of $M$.
	Let $\psi$ be the mapping that maps a quadrangulation $Q$ with a bipartition class $X$ to the map that has the vertices of $X$ as its vertices, and for each face of $Q$ the map $\psi(Q)$ has an edge between the two vertices of $X$ in that face.
	The mappings $\phi$ and $\psi$ are inverses, so they are bijections.
\end{proof}

With this bijection we can easily generate all plane maps with a given number $n$ of edges. First, generate all plane quadrangulations with $n+2$ vertices. These are exactly the quadrangulations with $n$ faces. Then check if the two colour classes are isomorphic, and use the bijection from Lemma \ref{lem:map_quad} to output one or two maps, depending on the number of non-isomorphic colour classes.

We used our program for generating quadrangulations to generate plane maps in this way, and compared the resulting counts to existing counts of plane maps. We used the general counts, counts of maps up to orientation-preserving automorphism and rooted counts found in \cite{rootedall, splitcounts, sumcounts}.

\section{Lopsp-operations}\label{sec:lopsp}
\subsection{From quadrangulations to lopsp-operations}
The predecoration of any lopsp-operation is an element of $\mathcal{Q}$. To generate all lopsp-operations of a given inflation factor $k$, we can therefore start by generating quadrangulations. Lemma \ref{lem:predeco_size} describes which quadrangulations must be generated.
\begin{lemma}\label{lem:predeco_size}
	Let $O$ be a lopsp-operation with inflation factor $k$. 
	\begin{itemize}
		\item If $k$ is even, then $v_1$ is not of colour 1 and the predecoration $Q(O)$ is a quadrangulation with $\frac{k+4}{2}$ vertices.
		\item If $k$ is odd, then $v_1$ is of colour 1 and the predecoration $Q(O)$ is a quadrangulation with $\frac{k+5}{2}$ vertices.
	\end{itemize}
\end{lemma}
\begin{proof}
	If $v_1$ is not of colour 1, then every vertex of colour 1 in $O$ has degree 4. As every face of $O$ has exactly one vertex of colour 1, it follows that $2\cdot k = |F_O| = 4|F_{Q(O)}|$ and therefore $k$ is even. If $v_1$ is of colour 1, then all vertices of colour 1 except $v_1$ have degree 4 and $v_1$ has degree 2. It follows that $2\cdot k =|F_O|= 4(|F_{Q(O)}| - 1) + 2 $ and therefore $k$ is odd.
	
	As $Q(O)$ is a quadrangulation, $4|F_{Q(O)}| = 2|E_{Q(O)}|$. It now follows from the Euler formula for $O(Q)$ that:
	
	\[ |V_{Q(O)}|= |E_{Q(O)}| - |F_{Q(O)}| + 2 = |F_{Q(O)}| + 2 =
	\begin{cases}
		\frac{k}{2} + 2  &= \frac{k + 4}{2}\quad \text{if $k$ is even}\\
		\frac{k-1}{2} + 1 + 2 &= \frac{k + 5}{2}\quad \text{if $k$ is odd}\\
	\end{cases}.\]
\end{proof}

Every quadrangulation is the predecoration of more than one lopsp-operation. To determine all the lopsp-operations that have a quadrangulation $Q$ as their predecoration we choose vertices $v_0$, $v_1$, and $v_2$ in $Q$ in every non-isomorphic way. In our program we have the automorphism group of $Q$, so that it is not difficult to determine which vertices are in different orbits under the automorphism group. If $k$ is odd, then it follows from Lemma \ref{lem:predeco_size} and the definition of a predecoration that $v_1$ must have degree 1 and its neighbour has colour 0. As $Q$ is bipartite, this determines the colours of all other vertices, so there is exactly one lopsp-operation for every choice of the $v_i$, considering that $v_1$ must have degree 1. If $k$ is even, then the colours are not determined yet. For each of the choices of $v_0$, $v_1$, and $v_2$ we then choose the colour of one vertex of the quadrangulation. It can be either 0 or 2, and both choices give rise to valid lopsp-operations that are different. The colours of every other vertex follow from this choice. This gives us all possible lopsp-operations with a certain inflation factor. 

\subsection{Connectivity}
A lopsp-operation is $\textbf{c}k$ if its `associated tiling' is $k$-connected and all faces have size at least $k$. We will not discuss the connection between lopsp-operations and tilings here. More details can be found in \cite{lopsp2021}. The definition of $\textbf{c}2$ and $\textbf{c}3$ operations we give in Definition \ref{def:2-3-connectivity} is equivalent to the one in \cite{lopsp2021}. The equivalence follows from results in that paper. 

\begin{definition}\label{def:2-3-connectivity}
	\begin{itemize}
		\item A lopsp-operation $O$ is $\textbf{c}2$ if for every cut-path $P$ of minimal length, there is no 2-cycle in $O_P$.
		\item A 4-cycle is non-trivial if it has a vertex of colour 0 on each side. A lopsp-operation $O$ is $\textbf{c}3$ if it is $\textbf{c}2$ and for every cut-path $P$ of minimal length, there is no non-trivial 4-cycle in a patch of two copies of $O_P$ sharing a copy of $P_{v_0,v_2}$ or two copies of $P_{v_0,v_1}$.
	\end{itemize}
\end{definition}

This definition cannot be checked immediately from the internal representation of a lopsp-operation in our program, as it relies on choosing cut-paths and cutting and gluing the lopsp-operation. Theorem \ref{thm:2conn_char} and Theorem \ref{thm:3conn_char} give characterisations of being $\textbf{c}2$ and $\textbf{c}3$ that can be checked knowing only the predecoration, $v_0$, $v_1$, and $v_2$, and whether $v_1$ is of colour 1.

In the following theorems we will use the concept of the `interior' of a face. Let $M'$ be a submap of a plane map $M$. Every vertex and edge of $M$ that is not in $M'$ is inside exactly one face of $M'$. The \emph{interior} of this face consists of all the vertices and edges of $M$ that are not in $M'$ that are inside the face.  

\begin{theorem}\label{thm:2conn_char}
	A lopsp-operation $O$ is $\textbf{c}2$ if and only if every 2-cycle in $Q(O)$ has one of the $v_i$ on each side.
\end{theorem}
\begin{proof}
	Assume first that there is a 2-cycle in the predecoration that does not have any of the $v_i$ on one side. Choose a cut-path $P$ in $O$. If the path passes through the interior of the 2-cycle, replace that part of $P$ by one of the edges of the 2-cycle. Let $P'$ be the resulting cut-path. The 2-cycle in $O$ now induces a 2-cycle in $O_{P'}$, so that $O$ is not $\textbf{c}2$ by Definition \ref{def:2-3-connectivity}.
	
	Now assume that $O$ is not $\textbf{c}2$, and that every 2-cycle in the predecoration has one of the $v_i$ on each side. There is a 2-cycle $c$ in $O_P$ for some cut-path $P$. This cycle induces a 2-cycle $c_O$ in $O$ that has none of the $v_i$ on one side, which we will call the inside. The other side is the outside. If $c_O$ is of colour 1, then it is also in the predecoration and we are done. Assume that $c_O$ is not of colour 1. Then one of its vertices, say $x$, is of colour 1. Let $y$ be the other vertex of $c_O$. As $v_0$ and $v_2$ are not of colour 1, $x$ is not one of those vertices. If $x$ would be $v_1$ then it would have degree 2 and it would therefore have two different neighbours of different colours, which is also not the case. Therefore $x$ is not one of the $v_i$. On every side of $c_O$ there is a neighbour of $x$. Let $v$ be the neighbour on the inside and let $w$ be the neighbour on the outside. As $O$ is a triangulation, $v$ and $w$ must be adjacent to $y$. If one of them, say $v$, has only one edge to $y$, then $v$ is the only vertex on that side of $c_O$ and it has degree 2. If that is the case for both $v$ and $w$, then $O$ has only four vertices. At most two of them can be $v_i$ as $x$ and $v$ are not, a contradiction. It follows that $v,y$ is a 2-cycle on the inside of $c_O$ or $w,y$ is a 2-cycle on the outside. In either case there is a 2-cycle of colour 1 that has none of the $v_i$ on one side, as neither $x$ nor any vertex on the inside of $c_O$ is one of the $v_i$, a contradiction. 
\end{proof}

\begin{theorem}\label{thm:3conn_char}
	A lopsp-operation $O$ is $\textbf{c}3$ if and only if for every submap $H$ of $Q(O)$ the following holds:
	\begin{itemize}
		\item Every face of $H$ of size 2 has either $v_0$ or $v_2$ in its interior, or $v_1$ is of colour 1 and it is the only vertex in the interior of that face.
		\item Every non-empty face of $H$ of size 4 has $v_0$, $v_1$ or $v_2$ in its interior.
	\end{itemize}
\end{theorem}

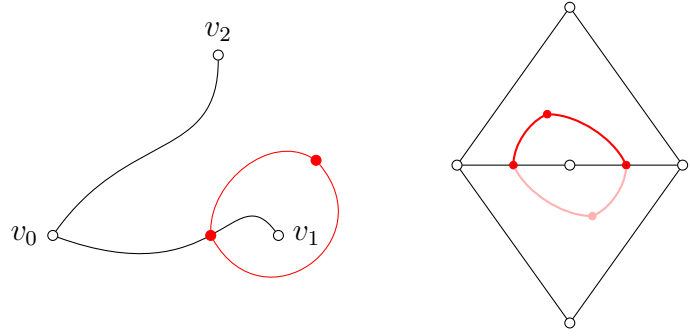
\begin{figure}
\begin{center}
		\begin{tikzpicture}[scale=1]

\tikzset{P/.style={shape=circle, draw=black, scale=0.4, fill=white}}
\tikzset{c/.style={shape=circle, draw=red, scale=0.4, fill=red}}

\coordinate (v0coo) at (-3,0);
\coordinate (v1coo) at (0,0);
\coordinate (v2coo) at (-0.8,2.4);
\coordinate (xcoo) at (-0.9,0);
\coordinate (ycoo) at (0.5,1);

\node[P,label=180:\large$v_0$] (v0) at (v0coo) {};
\node[P,label=0:\large$v_1$] (v1) at (v1coo) {};
\node[P,label=90:\large$v_2$] (v2) at (v2coo) {};

\begin{scope}[P]
\draw (v1) .. controls (-1.5,2) and (-2, -2) .. (v0);
\draw (v2) .. controls (-2,2.5) and (-5, 3.5) .. (v0);
\end{scope}

\node[c] (x) at (xcoo) {};
\node[c] (y) at (ycoo) {};

\draw[red] (x) edge[out = 90, in = 150] (y);
\draw[red] (x) edge[out = -60, in = -50, looseness = 2] (y);

\end{tikzpicture} \qquad 
	\begin{tikzpicture}[scale=3]
\tikzset{normal/.style={shape=circle, draw=black, scale=0.4, fill=white}}
\tikzset{P/.style={shape=circle, draw=red, scale=0.3, fill=red}}
\tikzset{P'/.style={shape=circle, draw=red!30!white, scale=0.3, fill=red!30!white}}

\node[normal] (v1) at (0, 0) {};
\node[normal] (v2) at (0, 0.7) {};
\node[normal] (v2') at (0, -0.7) {};
\node[normal] (v0) at (-0.5, 0) {};
\node[normal] (v0') at (0.5, 0) {};

\node[P] (x) at (-0.25,0) {};
\node[P] (y) at (0.25,0) {};
\node[P] (z) at (-0.1,0.226) {};
\node[P'] (z') at (0.1,-0.226) {};

\begin{scope}
\draw (v0) -- (v2) -- (v0') -- (v2') -- (v0) --(x) -- (x) -- (v1) -- (y) -- (y) -- (v0');
\end{scope}

\begin{scope}[thick]
\draw[red] (x) edge[bend left, looseness=0.8] (z);
\draw[red] (z) edge[bend left, looseness=0.8] (y);
\draw[red!30!white] (x) edge[bend right, looseness=0.8] (z');
\draw[red!30!white] (z') edge[bend right, looseness=0.8] (y);
\end{scope}

\end{tikzpicture}
\end{center}
	\caption{On the left, a schematic representation of a lopsp-operation is shown. The black lines represent a cut-path. A 2-cycle is shown in red. On the right, two double chambers sharing their copy of $v_1$ are shown for the lopsp-operation on the left with the given cut-path. The 2-cycle in the lopsp-operation induces a 4-cycle in the two double chambers.}
	\label{fig:2-cycle_to_4-cycle}
\end{figure}
\begin{proof}
	Assume that there is a submap $H$ of $Q(O)$ for which the condition described in the theorem does not hold. There are two cases:
	\begin{itemize}
		\item $H$ has a face $f$ of size 2 that does not contain $v_0$ or $v_2$ in its interior and if it contains $v_1$ then there is at least one other vertex in the interior of $f$. It follows immediately that if the interior of $f$ contains none of the $v_i$ --- which is not possible if $H$ is just two vertices and one edge ---, then by Theorem \ref{thm:2conn_char}, $O$ is not $\textbf{c}2$ and therefore not $\textbf{c}3$. We can therefore assume that the interior of $f$ contains $v_1$ and at least one other vertex. Let $P$ be a cut-path of $O$. The face $f$ of size 2 induces a cycle of length 4 in two copies of $O_P$ sharing $v_1$, as shown in Figure~\ref{fig:2-cycle_to_4-cycle}. As there is a vertex in $f$ that is not of colour 1, there is also such a vertex in the interior of the cycle in two copies of $O_P$. It follows that $O$ is not $\textbf{c}3$.
		
		\item $H$ has a non-empty face $f$ of size 4 that contains none of the $v_i$. Let $P$ be a cut-path of $O$. If there are no edges of $P$ in the interior of $f$, then there is either a 2-cycle in $O_P$ or a non-trivial 4-cycle. In both cases $O$ is not $\textbf{c}3$. Assume that there is an edge $e$ of $P$ in the interior of $f$. As $v_1$ and $v_2$ are not in the interior of $f$, the path $P$ must cross over the boundary of $f$ at least twice. Let $P_e$ be the subpath of $P$ that contains $e$ such that only the endpoints $x$ and $y$ of $P_e$ are in the boundary of $f$. Every other vertex of $P_e$ is in the interior of $f$ and is therefore different from $v_0$, $v_1$, and $v_2$. If $x$ and $y$ are adjacent in $H$, we can replace $P_e$ in $P$ by that edge. The resulting path is still a cut-path and contains fewer edges in the interior of $f$ than $P$. If every edge in the interior of $f$ can be removed in this way, we find a new cut-path that has no edges in the interior of $f$ and by the previous argument that implies that $O$ is not $\textbf{c}3$. We can now assume that $x$ and $y$ are not adjacent in $H$. By the planarity of $O$ and the fact that $f$ has only length 4, it follows that there are no edges of $P\setminus P_e$ in the interior of $f$. Note that as $v_0$ is not in the interior of $f$, the path $P_e$ is either completely contained in $P_{v_0,v_1}$ or in $P_{v_0,v_2}$. Let $X$ be two copies of $O_P$ glued together such that the subpath containing $P_e$ is on the glued side. As there can be no 2-cycles in $O_P$, $f$ induces a 4-cycle in $X$ whose interior is isomorphic to that of $f$. The two vertices corresponding to $x$ and $y$ are on the glued side. The other two vertices are adjacent to both $x$ and $y$ and are therefore each in one of the copies of $O_P$. It follows that $O$ is not $\textbf{c}3$.
	\end{itemize}

	Conversely, assume that $O$ is not $\textbf{c}3$. Let $X$ be the patch of two copies of $O_P$ that contains a non-trivial 4-cycle for a cut-path $P$ of minimal length. Let $W$ be the facial walk of the non-empty face of the 4-cycle in $X$. We say that the boundary of the only face in $X$ that is not a triangle is the boundary of $X$, also written as $\partial X$. 
	Assume that there is a vertex $x$ of colour 1 in $W$. Its neighbours in $W$ have the same colour, as otherwise there would be an edge between its neighbours and the inside of $W$ would be two chambers sharing an edge so $W$ would be trivial. Assume that $x$ is in $\partial X$. As $x$ has degree 4 and both neighbours of $x$ in $W$ have the same colour and are in $X$, at least one of them is in $\partial X$. If the other neighbour of $x$ in $\partial X$ is not in $W$, then $P$ would not be of minimal length. The part of $P$ corresponding to the path of length 2 in $\partial X$ between the two neighbours of $x$ could be replaced by one edge. 
	It follows that the neighbours of $x$ in $W$ are in $\partial X$ as well. As $W$ only has length 4, it is in only one copy of $O_P$. We can let $X$ be the patch consisting of both copies of $O_P$ that contain $x$. The walk $W$ is contained in that patch and $x$ is not on the boundary.
	
	This proves that we can assume that $x$ is not in $\partial X$. 
	Then all neighbours of $x$ are in the patch, also the neighbour of $x$ that is on the outside of $W$. Replace $x$ in $W$ by this neighbour. This new walk also has length 4, is non-trivial, and it has one fewer vertex of colour 1. If there is another vertex of colour 1 in $W$, repeat this argument to get a walk $W_1$ of colour 1. 
	
	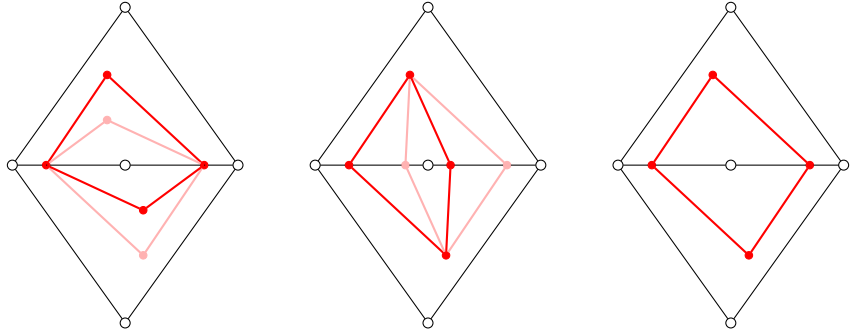
\begin{figure}
		\begin{center}
			\begin{tikzpicture}[scale=3]
	
	\tikzset{normal/.style={shape=circle, draw=black, scale=0.4, fill=white}}
	\tikzset{P/.style={shape=circle, draw=red, scale=0.3, fill=red}}
	\tikzset{P'/.style={shape=circle, draw=red!30!white, scale=0.3, fill=red!30!white}}
	
	\node[normal] (v1) at (0, 0) {};
	\node[normal] (v2) at (0, 0.7) {};
	\node[normal] (v2') at (0, -0.7) {};
	\node[normal] (v0) at (-0.5, 0) {};
	\node[normal] (v0') at (0.5, 0) {};
	
	\node[P] (x) at (-0.35,0) {};
	\node[P] (y) at (0.35,0) {};
	\node[P] (z) at (-0.08,0.4) {};
	\node[P'] (z') at (0.08,-0.4) {};
	\node[P] (w) at (0.08,-0.2) {};
	\node[P'] (w') at (-0.08,0.2) {};

	\begin{scope}
		\draw (v0) -- (v2) -- (v0') -- (v2') -- (v0) --(x) -- (x) -- (v1) -- (y) -- (y) -- (v0');
	\end{scope}
	
	\begin{scope}[thick]
		\draw[red] (x) edge[] (z);
		\draw[red] (z) edge[] (y);
		\draw[red!30!white] (x) edge[] (z');
		\draw[red!30!white] (z') edge[] (y);
		\draw[red] (x) edge[] (w);
		\draw[red] (w) edge[] (y);
		\draw[red!30!white] (x) edge[] (w');
		\draw[red!30!white] (w') edge[] (y);
	\end{scope}

\end{tikzpicture} \qquad 
			\begin{tikzpicture}[scale=3]
	\tikzset{normal/.style={shape=circle, draw=black, scale=0.4, fill=white}}
	\tikzset{P/.style={shape=circle, draw=red, scale=0.3, fill=red}}
	\tikzset{P'/.style={shape=circle, draw=red!30!white, scale=0.3, fill=red!30!white}}
	
	\node[normal] (v1) at (0, 0) {};
	\node[normal] (v2) at (0, 0.7) {};
	\node[normal] (v2') at (0, -0.7) {};
	\node[normal] (v0) at (-0.5, 0) {};
	\node[normal] (v0') at (0.5, 0) {};
	
	\node[P] (x) at (-0.35,0) {};
	\node[P'] (y) at (0.35,0) {};
	\node[P] (z) at (-0.08,0.4) {};
	\node[P] (z') at (0.08,-0.4) {};
	\node[P] (w) at (0.1,0) {};
	\node[P'] (w') at (-0.1,0) {};
	
	\begin{scope}
		\draw (v0) -- (v2) -- (v0') -- (v2') -- (v0) --(x) -- (w') -- (v1) -- (w) -- (y) -- (v0');
	\end{scope}
	
	\begin{scope}[thick]
		\draw[red] (x) edge[] (z);
		\draw[red!30!white] (z) edge[] (y);
		\draw[red] (x) edge[] (z');
		\draw[red!30!white] (z') edge[] (y);
		\draw[red] (z) edge[] (w);
		\draw[red!30!white] (z) edge[] (w');
		\draw[red] (z') edge[] (w);
		\draw[red!30!white] (z') edge[] (w');
	\end{scope}

\end{tikzpicture} \qquad
			\begin{tikzpicture}[scale=3]
	\tikzset{normal/.style={shape=circle, draw=black, scale=0.4, fill=white}}
	\tikzset{P/.style={shape=circle, draw=red, scale=0.3, fill=red}}

	\node[normal] (v1) at (0, 0) {};
	\node[normal] (v2) at (0, 0.7) {};
	\node[normal] (v2') at (0, -0.7) {};
	\node[normal] (v0) at (-0.5, 0) {};
	\node[normal] (v0') at (0.5, 0) {};
	
	\node[P] (x) at (-0.35,0) {};
	\node[P] (y) at (0.35,0) {};
	\node[P] (z) at (-0.08,0.4) {};
	\node[P] (z') at (0.08,-0.4) {};
	
	\begin{scope}
		\draw (v0) -- (v2) -- (v0') -- (v2') -- (v0) --(x) -- (v1) -- (y) -- (v0');
	\end{scope}
	
	\begin{scope}[thick]
		\draw[red] (x) edge[] (z);
		\draw[red] (z) edge[] (y);
		\draw[red] (x) edge[] (z');
		\draw[red] (z') edge[] (y);
	\end{scope}

\end{tikzpicture}
		\end{center}
		\caption{This figure shows the three possibilities for a 4-cycle around $v_1$ in two copies of a double chamber patch. The pink cycle shows the `reflection' of the red cycle. In the third case, the cycle is its own reflection.}
		\label{fig:4-cycles_aroundv1}
	\end{figure}

	The walk $W_1$ induces a closed walk $W_O$ of colour 1 in $O$. If $v_1$ is not on the inside of $W_1$, then $W_O$ is a non-trivial plane walk that has none of the $v_i$ on the inside. This is a contradiction with (ii). It follows that $v_1$ is on the inside of $W_1$, and as $W_1$ is non-trivial, it is either not of colour 1, or it is not the only vertex on the inside. As $X$ consists of two copies of $O_P$ sharing their copy of $v_1$, $X$ is point symmetric with respect to $v_1$. We can therefore consider the walk $W_1'$ in $X$ that is the `reflection' of $W_1$. The walks $W_1$ and $W_1'$ are either the same walk or they intersect in two vertices, as shown in Figure~\ref{fig:4-cycles_aroundv1}. If they are not the same walk, then we can choose a different walk consisting of edges of $W_1$ and $W_1'$ that has $v_1$ on the inside, is non-trivial, and is the same walk as its reflection. We can therefore assume that $W_1$ and $W_1'$ are the same walk. This walk induces a closed walk of length 2 in $O$ that has $v_1$ on the inside and either $v_1$ is not of colour 1 or there is at least one other vertex inside the walk. This is a contradiction with (i) as the walk is of colour 1 and therefore it is in the predecoration.
\end{proof}

\subsection{Lsp-operations}
Some lopsp-operations do not only preserve the orientation-preserving symmetries of maps, but also the orientation-reversing symmetries. These operations can be described as local symmetry-preserving operations (Definition~\ref{def:lsp}). Many well-known operations on polyhedra, such as truncation, ambo, and dual are lsp-operations. 

\begin{definition}\label{def:lsp}
	Let $O$ be a 2-connected plane map with vertex set $V$, together with a colouring $c: V \rightarrow \{0,1,2\}$. One of the faces is called the outer face. This face contains three special vertices marked as
	$v_0$, $v_1$, and $v_2$. We say that a vertex $v$ has \emph{colour} $i$ if $c(v)=i$. 
	This 3-coloured map $O$ is a \emph{local symmetry preserving operation}, lsp-operation for short, if the following properties hold:
	\begin{enumerate}
		\item Every inner face --- i.e.\ every face that is not the outer face --- is a triangle.
		\item There are no edges between vertices of the same colour.
		\item For each vertex that is not in the outer face:
		\begin{align*}
			c(v)=1 &\Rightarrow deg(v)=4\\
		\end{align*}
		For each vertex $v$ in the outer face, different from $v_0$, $v_1$, and $v_2$:
		\begin{align*}
			c(v)=1 &\Rightarrow deg(v)=3\\
		\end{align*}
		and
		\[c(v_0),c(v_2)\neq 1\]
		\begin{align*}
			c(v_1)=1 &\Rightarrow deg(v_1)=2\\
		\end{align*}
	\end{enumerate}
\end{definition}

An lsp-operation is applied similarly to a lopsp-operation, except that there is no need to choose a cut-path, and instead of gluing copies into $B_M$ with the edges of colour 0 removed, copies of $O$ or the mirror image of $O$ are glued into $B_M$. To get the lopsp-operation corresponding to an lsp-operation, a mirrored copy of the lsp-operation is glued into the outer face, as proved in \cite{lopsp2021}. A lopsp-operation $O$ can therefore be written as an lsp-operation if and only if there is a cycle in $O$ that contains $v_0$, $v_1$, and $v_2$ and one side of the cycle is the mirror image of the other side. In Theorem \ref{thm:lopsp_is_lsp} we prove another equivalence that is often easier to check. There are results on the fixpoints of orientation-reversing automorphisms of simple triangulations in the literature (see \cite{kang_ORautomorphism, tutte_ORautomorphism}), but as lopsp-operations may have parallel edges we could not use these directly. We did use those proofs as an inspiration for the proof of Theorem \ref{thm:lopsp_is_lsp}.

\begin{theorem}\label{thm:lopsp_is_lsp}
	A lopsp-operation can be written as an lsp-operation if and only if it has an orientation-reversing automorphism that fixes $v_0$, $v_1$, and $v_2$.
\end{theorem}
\begin{proof}
	Assume that a lopsp-operation can be written as an lsp-operation $O$. The lopsp-operation can be obtained from the lsp-operation by gluing a mirror copy of $O$ into the outer face of $O$. There is clearly an orientation-reversing automorphism that fixes the boundary of the outer face of the lsp-operation and switches the sides of that cycle.
	
	Conversely, assume that a lopsp-operation $O$ has an orientation-reversing automorphism $\phi$ that fixes $v_0$, $v_1$, and $v_2$. This automorphism can be extended in a unique way to an automorphism of the barycentric subdivision $B_O$ of $O$. It follows from the fact that $O$ is 2-connected and has no loops that $B_O$ is a simple graph. Let $H$ be the submap of $B_O$ consisting of all vertices and edges of $B_O$ that are fixed by $\phi$. 
	
	Let $x$ be any vertex in $H$. Its degree in $O$ is even, and the colours of its neighbours alternate. Let $e$ be any edge incident with $x$. As $x$ is a fixpoint of $\phi$, the edge $\phi(e)$ is also incident with $x$. Let $e = e_0, e_1,\ldots, e_{k-1}, \phi(e) = e_k, e_{k+1},\ldots, e_n$ be the cyclic order of edges around $x$. As the colours of the edges incident with a vertex alternate, and the colours of $e$ and $\phi(e)$ are the same, $k$ is even. As $\phi$ is orientation-reversing, $e_{\frac{k}{2}}$ and $e_{\frac{n+k+1}{2}}$ and no other incident edges of $x$ are fixed by $\phi$. It follows that every vertex in $H$ has degree 2, so $H$ consists of a number of disjoint cycles. The two incident edges that are fixed are always exactly opposite each other in the rotational order.
	
	Let $c$ be the cycle in $H$ that contains $v_0$, and let $x$ be a vertex of $B_O$ that is not in this cycle. The graph $B_O$ is connected, so there is a path $P$ without vertices in $c$ from $x$ to a vertex $y$ that is adjacent to a vertex in $c$. As $\phi$ is orientation reversing, $\phi(y)$ is on the other side of $c$ than $y$ and $x$. The path $\phi(P)$ goes from the vertex $\phi(x)$ to the vertex $\phi(y)$ and does not contain vertices in $c$, so by the Jordan curve theorem $\phi(x)$ is on the other side of $c$ than $x$. A consequence of this is that the vertices of $c$ are the only vertices of $B_O$ that are fixed by $\phi$, and therefore $v_0$, $v_1$, and $v_2$ are in $c$.
	
	Assume that there is a vertex of colour 2 (in $B_O$) in $c$. This vertex corresponds to a chamber, so it has degree 6 in $B_O$. The two edges of $c$ incident to it must have different colours, as they are opposite each other. That implies that $\phi$ fixes one vertex of the chamber, and switches the other two. This is impossible as $\phi$ maps vertices of $O$ to vertices of the same colour. It follows that as vertices of $B_O$, there are no vertices of colour 2 in $c$. 
	This implies that $c$ induces a cycle $c_O$ in $O$. The vertices in this cycle, as vertices of $O$, can have colour 2. It follows from our previous observations that this cycle together with the interior of one of its sides is an lsp-operation that represents the same operation as $O$.
\end{proof}

\subsection{Computational results}

\begin{table}
	\centering
	\begin{tabular}{|c|rrrrrr|}
		\hline
		inflation&  \multicolumn{2}{|c|}{all} &  \multicolumn{2}{|c|}{$\textbf{c}2$} & \multicolumn{2}{|c|}{$\textbf{c}3$} \\ 
		\cline{2-7}
		factor & \multicolumn{1}{c}{all} & \multicolumn{1}{c|}{lsp} &  \multicolumn{1}{c}{all} & \multicolumn{1}{c|}{lsp} &  \multicolumn{1}{c}{all} & \multicolumn{1}{c|}{lsp} \\ 
		\hline 
		1&  2& 2& 2& 2& 2& 2 \\
		2&  6& 6& 6& 6& 2& 2 \\
		3&  12& 12& 8& 8& 4& 4 \\
		4& 54 &54&30&30&6	&6\\
		5&86&64&38&34&8&4\\
		6&	466&392&154&140&20&20\\	
		7&730&380&194&148&30&20\\
		8&	4 182&2 694&810&630&62&54\\
		9&6 828&2 316&1 034&638&102&64\\
		10&39 624&18 012&4 386&2 766&198&144\\
		11&68 402&14 332&5 732&2 728&318&132\\
		12&	 395 976		    &118 356		&24 528		&11 928		&664		&404\\
		13	    &718 116		    &89 040		&32 908		&11 552		&1 122		&396\\
		14	    &4 137 048	    &768 312		&141 642		&50 724		&2 272		&1 112\\
		15	    &7 782 846	    &554 124		&194 800		&48 512		&3 982		&1 100\\
		16	    &44 691 306	    &4 942 542	&842 118		&213 294		&7 914		&2 958\\
		17	    &86 164 178	    &3 448 216	&1 182 654	&202 158		&14 254		&2 768\\
		18	    &494 280 434	    &31 573 408	&5 131 120	&888 590		&28 296		&7 972\\
		19	    &967 710 288	    &21 445 656	&7 325 268	&836 736		&52 400		&7 560\\
		20	    &5 555 113 656	    &200 570 268	&31 881 474	&3 672 918	&103 028		&21 300\\
		21	    &10 976 285 688	&133 257 288	&46 073 266	&3 442 202	&194 794		&20 076\\
		22	    &63 116 746 584	&1 268 330 664	&201 076 524	&15 079 608	&380 864		&56 296\\
		23	    &125 378 761 834	&827 198 660	&293 139 672	&14 083 824	&731 622		&52 380\\
		24	    &		        &		    &1 282 452 080	&61 548 760	&1 426 556	&148 956	\\
		25	    &		        &		    &		    &		    &2 773 348	&138 384	\\
		\hline
	\end{tabular} 
	\caption{This table shows the number of lopsp-operations that our program counted for inflation factor up to 25. Note that operations that are $\textbf{c}3$ are also $\textbf{c}2$, so they are also included in those counts.}
	\label{tab:lopsp_counts}
\end{table}

Table~\ref{tab:lopsp_counts} contains the numbers of lopsp-operations with different properties we have obtained with the new program, which can be found at \url{https://github.com/hvdncamp/lopsp_tools}. It is interesting to see that there are relatively few $\textbf{c}3$ operations. For inflation factors 1 and 2 all lopsp-operations are $\textbf{c}3$, but the higher the inflation factor, the fewer lopsp-operations are $\textbf{c}3$. For inflation factor 20 only 0.0018\% of all lopsp-operations is $\textbf{c}3$. For this reason some optimisations have been added to the program that are used when only $\textbf{c}2$ or $\textbf{c}3$ operations are required. With these optimisations we were able to obtain counts of $\textbf{c}2$ and $\textbf{c}3$ operations for higher inflation factors.

An important optimisation that has been added is that if an intermediary quadrangulation in the generation process has 4 vertices of degree 1, then it is not extended further. This is because the generation process ensures that once a quadrangulation has 3 vertices of degree 1, its number of vertices of degree 1 cannot decrease by applying canonical extensions. The predecoration of a $\textbf{c}2$ operation has at most 3 vertices of degree 1, so quadrangulations with more than three vertices of degree 1 cannot lead to valid $\textbf{c}2$ or $\textbf{c}3$ operations and therefore they are not extended further.

Once a quadrangulation is generated, we first do some quick checks to exclude quadrangulations that clearly cannot be the predecoration of $\textbf{c}2$ or $\textbf{c}3$ lopsp-operations. For example, if a quadrangulation has four edges between the same two vertices, it is rejected. In that case there are four 2-cycles whose interiors are disjoint, so by Theorem~\ref{thm:2conn_char} that quadrangulation is not the predecoration of any $\textbf{c}2$ lopsp-operation. In this way we can drastically reduce the number of times we have to check if a lopsp-operation is $\textbf{c}2$ or $\textbf{c}3$, which is an expensive check. For example, we let the program generate all $\textbf{c}3$ lopsp-operations of inflation factor 20. Because of the pruning during the generation of quadrangulations, only 598 628 of the 2 152 298 quadrangulations are generated. Then after filtering out the obvious cases, only 17 940 quadrangulations are left that are candidates to be predecorations. Out of these 17 940 quadrangulations, 9 867 are predecorations of $\textbf{c}3$ lopsp-operations.

Note that the difference between the number of lopsp-operations of inflation factor $k$ and $k-1$ is larger when $k$ is even than when $k$ is odd. This is most obvious for lsp-operations. The number of those operations often decreases when the inflation factor is increased by one. Recall that by Lemma~\ref{lem:predeco_size}, lopsp-operations of inflation factors $k$ and $k+1$ are generated from the same set of quadrangulations if $k$ is odd. As the choice of $v_1$ is much more restricted if the inflation factor is odd, there are much fewer operations of inflation factor $k$ than of $k+1$, if $k$ is odd.

\begin{remark}
	With 
	\begin{alignat*}{3}
		&n^{tot} &&=  \text{ Number of lopsp-operations}\\
		&n^{op} &&=  \text{ Number of lopsp-operations up to orientation-preserving automorphisms}\\
		&n^{chir} &&=  \text{ Number of lopsp-operations without orientation-reversing automorphisms}\\
		&n^{lsp} &&=  \text{ Number of lopsp-operations that can be written as lsp-operations}
	\end{alignat*}

	we have the following identity, which also hold when restricted to lopsp-operations of a certain inflation factor or connectivity:
	
	\[n^{tot} = n^{lsp} + n^{chir}.\]
	
	It follows that 
	
	\[n^{op}= n^{lsp} + 2\cdot n^{chir} = 2\cdot n^{tot} - n^{lsp}.\]
	
	This proves that $n^{op}$ can be calculated from $n^{tot}$ and $n^{lsp}$. Therefore we do not list it in Table~\ref{tab:lopsp_counts}.
\end{remark}

As these are the first published counts of lopsp-operations, it was not possible to compare them to existing counts. We did however compare the counts of $\textbf{c}3$ lsp-operations to those from \cite{lspcounts}. The general and $\textbf{c}2$ lsp-counts from that paper cannot be compared, as they use a slightly different definition of lsp-operations. We also wrote a program that generates all lopsp-operations by generating all triangulations with plantri (\cite{plantri_program}), filtering them, and choosing $v_0$, $v_1$, and $v_2$ in all possible ways. We checked that these programs give the same results up to inflation factor 12. To check that our program determines whether an operation is $\textbf{c}2$ or $\textbf{c}3$ correctly, we wrote a slower but simpler program that determines whether an operation is $\textbf{c}2$ or $\textbf{c}3$ and we compared the results up to inflation factor 15.

\section{Future work}

For the generation of $\textbf{c}2$ and $\textbf{c}3$ operations, our program is not optimal. Filters during the generation process improve the performance in these cases, but too many quadrangulations are still generated and it needs to be checked whether operations are $\textbf{c}3$ or not. Developing a faster program for $\textbf{c}2$ or $\textbf{c}3$ operations should be possible, but it would require a different approach.

An interesting application of this program could be the generation of symmetric maps. As lopsp-operations preserve symmetry, maps with certain symmetry goups can be generated by applying lopsp-operations to a base map with that symmetry group. To find out whether our program can be used to generate all maps with a certain symmetry group, two things must be proved. The first is the existence of a map with that symmetry group such that every other map with the same symmetry can be generated by applying a lopsp-operation to the base map. The second is that one map cannot be obtained by applying different lopsp-operations to a base map. For example: The tetrahedron can be obtained by either applying the identity or the dual to the tetrahedron. The problem here is that the tetrahedron is self-dual. Could something similar happen for other symmetry groups or can it only happen for self-dual base maps? 

\nocite{*}
\bibliographystyle{abbrvnat}
% use the following instead if you encounter problems 
%\bibliographystyle{alpha}
\bibliography{sources}
\label{sec:biblio}
%\bibliographystyle{plain}
%\bibliography{sources}

\end{document}